\documentclass[11pt]{articulofederico}
\usepackage{graphicx}
\usepackage{dcolumn}
\usepackage{enumerate}
\usepackage{subfigure}
\usepackage{epsfig}
\usepackage{amssymb}
\usepackage{epsf} 	
\usepackage{amscd}

\usepackage[utf8]{inputenc}
\usepackage[spanish, activeacute]{babel}

\usepackage{xcolor} 

\usepackage{amsfonts} 

\usepackage{amssymb}
\usepackage{enumerate}
\usepackage{graphicx}

\usepackage{amsmath,amssymb,enumerate,amsthm,flafter}

\newtheorem{teor}{Teorema}[section]
\newtheorem{prop}[teor]{Proposici\'on}
\newtheorem{cor}[teor]{Corolario}
\newtheorem{lem}[teor]{Lema}

\theoremstyle{definition}
 
  \newtheorem{defi}[teor]{Definici\'on}
    
        \newtheorem{Observaci'on}[teor]{Observaci\'on}

 \newtheorem{ej}[teor]{Ejemplo}

\renewcommand\figurename{Figura}

\DeclareMathOperator{\reg}{r}
\DeclareMathOperator{\ac}{b}

\newcommand{\F}{{\ensuremath{\mathbb{F}}}}
\newcommand{\R}{{\ensuremath{\mathbb{R}}}}
\newcommand{\Z}{{\ensuremath{\mathbb{Z}}}}
\newcommand{\N}{{\ensuremath{\mathbb{N}}}}
\newcommand{\Cx}{{\ensuremath{\mathbb{C}}}}
\newcommand{\A}{{\ensuremath{\mathcal{A}}}}
\newcommand{\Shi}{{\ensuremath{\mathcal{S}}}}
\newcommand{\B}{{\ensuremath{\mathcal{B}}}}
\newcommand{\C}{{\ensuremath{\mathcal{C}}}}
\newcommand{\Lin}{{\ensuremath{\mathcal{L}}}}
\newcommand{\Pk}{{\mathcal{P}}}
\newcommand{\rank}{{\mathrm{rango}}}

\begin{document}
\title{\textsf{Tres lecciones en combinatoria algebraica.  \\
\normalsize{III. Arreglos de hiperplanos.}}}

 \author{\textsf{Federico Ardila\footnote{\textsf{San Francisco State University, San Francisco, CA, USA y Universidad de Los Andes, Bogot\'a, Colombia, federico@sfsu.edu -- financiado por la CAREER Award DMS-0956178 y la beca DMS-0801075 de la National Science Foundation de los Estados Unidos, y por la SFSU-Colombia Combinatorics Initiative.}}}\qquad
\textsf{Emerson Le\'on\footnote{\textsf{Freie Universit\"at Berlin, Alemania,  emerson@zedat.fu-berlin.de -financiado por el Berlin Mathematical School.}}}\\
\textsf{Mercedes Rosas\footnote{\textsf{Universidad de Sevilla, Espa\~na, mrosas@us.es -- financiada por los proyectos MTM2007--64509 del Ministerio de Ciencias e Innovaci\'on de Espa\~na y FQM333 de la Junta de Andalucia.}}}\qquad
\textsf{Mark Skandera\footnote{\textsf{Lehigh University, 
Bethlehem, PA, USA, 
mas906@math.lehigh.edu -- financiado por la beca H98230-11-1-0192 de la National Security Agency de los Estados Unidos.}}}
}
\date{}
\maketitle

\begin{abstract} 
En esta serie de tres art\'\i culos, damos una exposici\'on de varios resultados y problemas abiertos en tres \'areas de la combinatoria algebraica y geom\'etrica:  las matrices totalmente no negativas,
las representaciones
del grupo sim\'etrico $S_n$, y los arreglos de hiperplanos. Esta tercera parte presenta una introducci\'on a los arreglos de hiperplanos desde un punto de vista combinatorio.
\end{abstract}



En marzo de 2003 se llev\'o a cabo el Primer Encuentro Colombiano de Combinatoria en Bogot\'a, Colombia. Como parte del encuentro, se organizaron tres minicursos, dictados por Federico Ardila, Mercedes Rosas, y Mark Skandera. Esta serie resume el material presentado en estos cursos en tres art\'{\i}culos:   \emph{I. Matrices
totalmente no negativas y funciones sim\'etricas} \cite{ALRS1}, 
\emph{II. Las funciones sim\'etricas y la teor\'{\i}a de las representaciones} \cite{ALRS2}, y \emph{III. Arreglos de hiperplanos.} \cite{ALRS3} 

%

En esta tercera parte presenta una introducci'on a los arreglos de hiperplanos desde un punto de vista combinatorio. Estudiaremos las regiones, el polinomio caracter'istico, y el poset de intersecciones de un arreglo. Tambi\'en estudiaremos algunos arreglos especiales que est\'an relacionados con objetos combinatorios cl\'asicos.

\section{\textsf{Arreglos de rectas y arreglos de planos.}}

La siguiente pregunta servir\'a como motivaci\'on:
Si trazamos $n$ l\'{\i}neas rectas en el plano $\R^2$, ¿cu\'al es el mayor n\'umero de regiones que podemos formar?

Es f\'acil ver que para lograr que el n\'umero de regiones sea m\'aximo, las rectas deben estar en \emph{posici\'on general}; es decir:



\begin{itemize}
\item No hay dos rectas paralelas.
\item No hay tres rectas concurrentes.
\end{itemize}

En efecto, si algunas rectas del arreglo no cumplen estas propiedades, podemos moverlas un poco de manera que el n\'umero de regiones aumente. Lo sorprendente es que las dos condiciones anteriores son suficientes, y determinan de manera \'unica el n\'umero de regiones que se forman.

\begin{teor}\label{2D}
Cualquier arreglo de  $n$ rectas en posici\'on general en $\R^2$ tiene el m'aximo n'umero de regiones $\reg(n)$ y de regiones acotadas $\ac(n)$ entre todos los arreglos de $n$ rectas. Estos valores m'aximos est'an dados por las f'ormulas
\begin{eqnarray*}
\reg(n) &=& {n \choose 2} +  {n \choose 1} + {n \choose 0},\\
\ac(n) &=& {n \choose 2} - {n \choose 1} + {n \choose 0}.
\end{eqnarray*}
\end{teor}

\begin{proof}

Demostremos por inducci\'on que los n\'umeros $\reg(n)$ y $\ac(n)$ de regiones y de regiones acotadas determinadas por $n$ rectas gen\'ericas dependen s\'olo de $n$, y satisfacen las recurrencias
\begin{equation}\label{eq1}\reg(n)=\reg(n-1)+n, \qquad \qquad \ac(n)=\ac(n-1)+(n-2),\end{equation}
El caso $n=0$ es claro. Consideremos $n$ rectas en posici\'on general. Si borramos una de ellas, las otras $n-1$ dividen al plano en $\reg(n-1)$ regiones, $\ac(n-1)$ de las cuales son acotadas. Al volver a introducir la recta borrada, \'esta dividir\'a a algunas de estas regiones en dos. Las $n-1$ rectas dividen a la nueva recta en $n$ sectores, de las cuales $n-2$ son acotados. 
Por lo tanto, la nueva recta divide a $n$ de las $\reg(n-1)$ regiones, y exactamente $n-2$ de \'estas 
introducen una nueva regi\'on acotada. Las f\'ormulas expl\'{\i}citas para $\reg(n)$ y $\ac(n)$ se siguen.
\end{proof}
\begin{figure}[htb]\centering
\includegraphics[width=4cm]{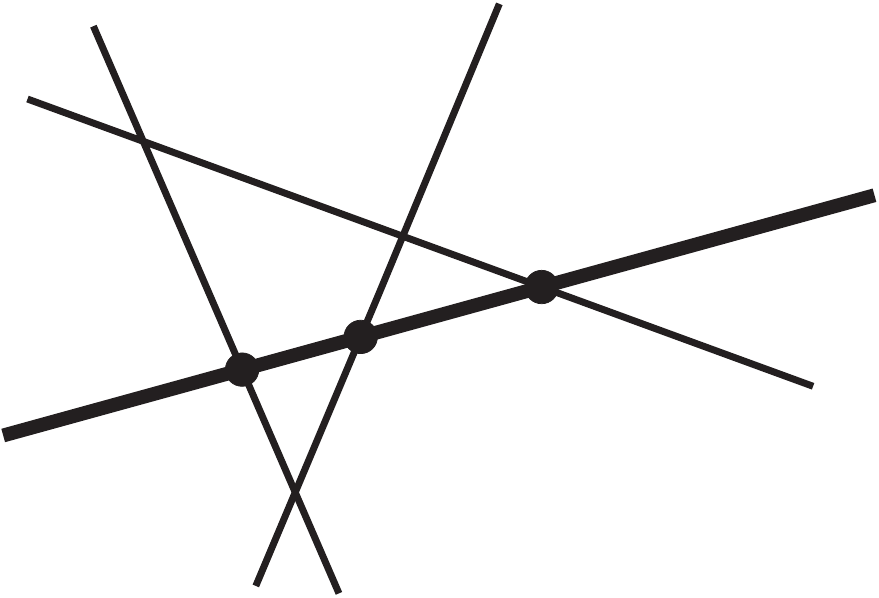}
\caption{El efecto de introducir una nueva recta al arreglo. }\label{rectas}
\end{figure}
%

Ahora generalizaremos el resultado anterior a tres dimensiones: 

Si trazamos $n$ planos en $\R^3$, ¿cu\'al es el mayor n\'umero de regiones que podemos formar?

Igual que en el caso anterior, es claro que para obtener el m\'aximo n\'umero de regiones es necesario que los planos est\'en en \emph{posici\'on general}, en el siguiente sentido:

\qquad $\bullet$ No hay dos planos paralelos.

\qquad$\bullet$  No hay tres que se corten en una linea recta, ni tampoco tres que no se corten.

\qquad$\bullet$  No hay cuatro planos que pasen por un mismo punto.


Nuevamente, las tres condiciones anteriores determinan de manera \'unica el n\'umero de regiones en las que se divide el espacio, y tambi\'en el n\'umero de regiones acotadas, para cada valor de $n$. En este caso, llamaremos $\reg_3(n)$ al n\'umero de regiones que son determinadas por $n$ planos en posici\'on general, y $\ac_3(n)$ el n\'umero de regiones acotadas.

\begin{figure}[htb]\centering
\includegraphics[width=6cm]{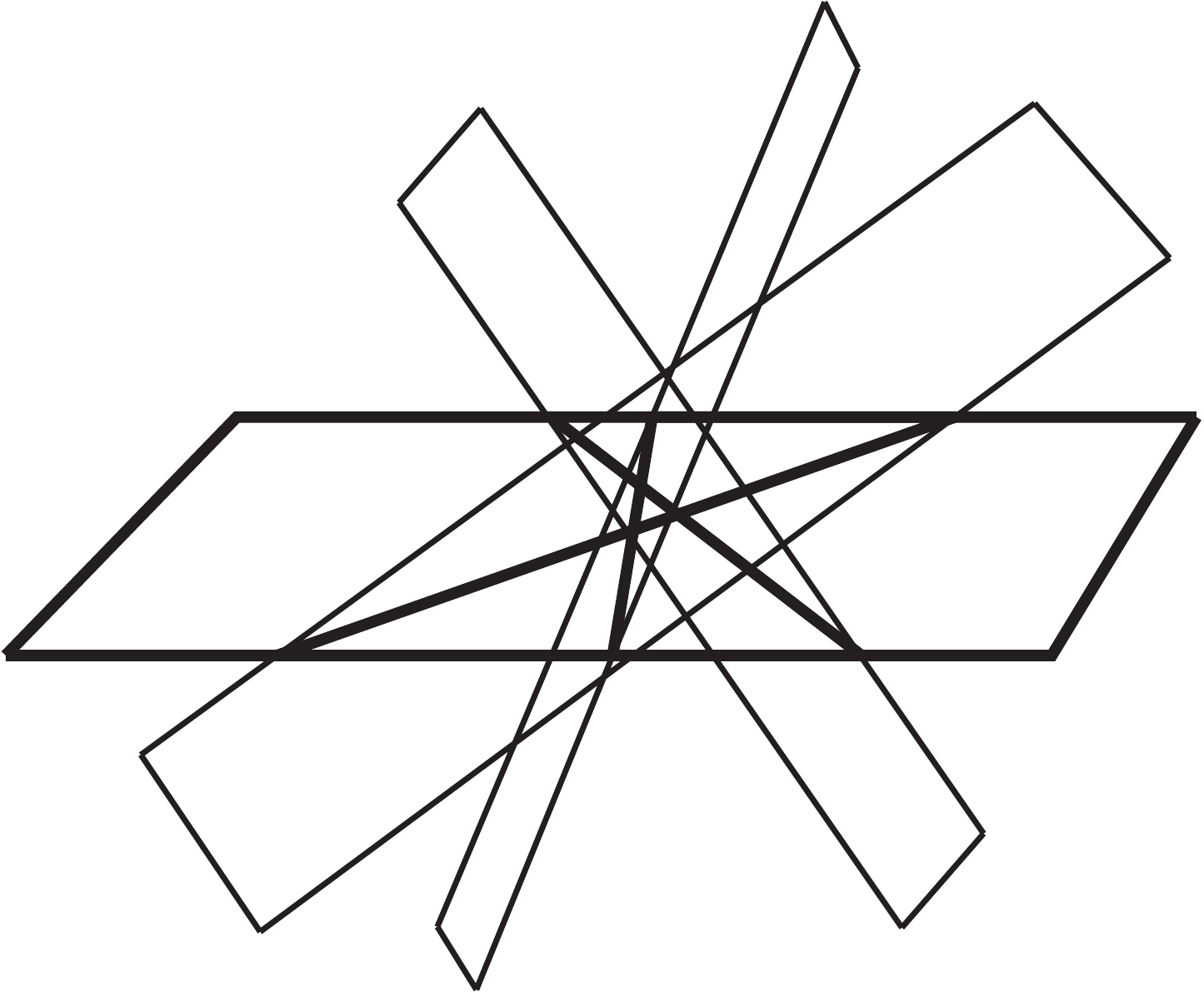}
\caption{Planos en $\R^3$ y su intersecci\'on en $P$.}\label{planos}
\end{figure}

\begin{teor}\label{3D}
Cualquier arreglo de  $n$ planos en posici\'on general en $\R^3$ tiene el m'aximo n'umero de regiones $\reg_3(n)$ y de regiones acotadas $\ac_3(n)$ entre todos los arreglos de $n$ planos. Estos n\'umeros son
\begin{eqnarray*}
\reg_3(n) &=& {n \choose 3} + {n \choose 2} +  {n \choose 1} + {n \choose 0}, \\
\ac_3(n) &=& {n \choose 3} - {n \choose 2} + {n \choose 1} - {n \choose 0}.
\end{eqnarray*}
\end{teor}

\begin{proof}
Nuevamente consideramos un plano especial $P$ de un arreglo gen\'erico de $n$ planos, para obtener relaciones recursivas para $\reg_3(n)$ y $\ac_3(n)$.
Quitando este plano del arreglo, obtenemos un arreglo con $n-1$ planos en posici\'on general, que genera $\reg_3(n-1)$ regiones, entre las cuales hay $\ac_3(n-1)$ regiones acotadas.


Los otros $n-1$ planos determinan $n-1$ rectas en $P$. Si los planos est\'an en posici\'on general, las rectas en $P$ se encuentran en posici\'on general. Por lo tanto ellas forman ${n\choose 2}+{n\choose 1}+{n\choose 0}$ regiones, de las cuales ${n\choose 2}-{n\choose 1}+{n\choose 0}$ son acotadas. Al volver a colocar el plano $P$, el n\'umero de estas regiones que son divididas en dos es igual al n\'umero de regiones que se forman en $P$; es decir, $r_2(n-1):= r(n-1) = {n-1 \choose 2}+{n-1 \choose 1}+{n-1 \choose 0}$. Cada regi\'on acotada en $P$ genera una nueva regi\'on acotada.
 As\'{\i} obtenemos que
\begin{equation}
\begin{split}\reg_3(n)&=\reg_3(n-1)+\reg_2(n-1)=\reg_3(n-1)+{n-1\choose 2}+{n-1\choose 1}+{n-1\choose 0},\\
\ac_3(n)&=\ac_3(n-1)+\ac_2(n-1)=\ac_3(n-1)+{n-1\choose 2}-{n-1\choose 1}+{n-1\choose 0}. \end{split}
\end{equation}
De las recurrencias anteriores se obtienen las f\'ormulas deseadas.
\end{proof}

Estos dos ejemplos sugieren una generalizaci\'on natural en cualquier dimensi\'on. Es natural conjeturar que cualquier arreglo de $n$ hiperplanos gen\'ericos en $\R^d$ tiene un cierto n\'umero de regiones $\reg_d(n)$ y regiones acotadas $\ac_d(n)$, y es f\'acil adivinar cu\'ales son esos n\'umeros. 
Esta generalizaci\'on es correcta, y la demostraremos en el Teorema \ref{nD}. Pero antes de hacerlo, debemos precisar el significado de los diferentes conceptos involucrados.
Debemos saber qu\'e son los hiperplanos, las regiones, y qu\'e significa estar en posici\'on general en el caso de dimensi\'on $d$. En la siguiente secci\'on introducimos estas definiciones.

\section{\textsf{Arreglos de hiperplanos}}

Si $v=(v_1,\ldots,\,v_d)$ y $x=(x_1,\ldots,x_n)$ son vectores en $\R^d$, denotamos el producto punto de $v$ y $x$ por $v\cdot x=v_1x_1+v_2x_2+\cdots +v_dx_d$. 

\begin{defi}\label{hiper}
Un \emph{hiperplano} en $\R^d$ es un conjunto de la forma $$H=\{x\in \R^d:v\cdot x=a\},$$ donde  $v=(v_1,\ldots,\,v_d)\in \R^d$,  $a\in\R$, y no todos los $v_i$ son iguales a cero. Un \emph{arreglo de hiperplanos} $\A$ es una colecci\'on finita de hiperplanos en $\R^d$.\end{defi}

\begin{defi}
Un \emph{subespacio af\'{\i}n} es un conjunto $S\subseteq\R^d$, tal que si $x,y \in S$  y $\lambda\in\R$, entonces $\lambda x+ (1-\lambda)y\in S$.\end{defi}

Si $S$ es un subespacio af\'{\i}n y $x\in S$, entonces la translaci\'on $S-x$ es un subespacio vectorial de $\R^d$. Esto nos permite definir la \emph{dimensi\'on} de $S$ como la dimensi\'on de $S-x$ como espacio vectorial.
Se encuentra por ejemplo que la dimensi\'on de un hiperplano $H$ es $d-1$ (pues su translaci\'on resulta ser el espacio ortogonal al vector $v$, tomando a $H$ como en la Definici\'on \ref{hiper}). Toda intersecci\'on de hiperplanos forma un subespacio af\'{\i}n, y todo subespacio af\'{\i}n de dimensi\'on $k$ se puede expresar como la intersecci\'on de $d-k$ hiperplanos en $\R^d$. 

Cada  hiperplano $H=\{x\in \R^d:v\cdot x=a\}$ divide a $\R^d$ en dos regiones, donde $v\cdot x < a$ y $v\cdot x > a$ respectivamente.

Dado un arreglo de hiperplanos $\A$, este divide al espacio en varias componentes conexas, llamadas \emph{regiones}. 
Si $\A$ est\'a formado por $n$ hiperplanos 
\[
H_i=\{x \in \R^d:v_i\cdot x=a_i\},
\]

donde  $v_i\in \R^d$ para  $1 \leq i \leq n$, cada \emph{regi\'on} de $\A$ puede ser descrita con un sistema de desigualdades {\bf que tiene soluci\'on} en $\R^d$, donde se selecciona una desigualdad de la forma 
\[
v_i\cdot x< a_i\textrm{\qquad o \qquad}v_i\cdot x> a_i
\]
para cada $1 \leq i \leq n$. 




\begin{defi}
Un arreglo de hiperplanos se encuentra en \emph{posici\'on general} si se cumplen las siguientes condiciones:

\qquad$\bullet$
 Dos hiperplanos distintos siempre se intersectan.

\qquad$\bullet$
Tres hiperplanos siempre se intersectan en un subespacio af\'{\i}n de dimensi\'on $d-3$.

\qquad$\bullet$
 Cuatro hiperplanos siempre se intersectan en un subespacio af\'{\i}n de dimensi\'on $d-4$.

\hspace{1.3cm}\vdots

\qquad$\bullet$
 $d$ hiperplanos siempre se intersectan en un punto.

\qquad$\bullet$
 $d+1$ hiperplanos nunca se intersectan.

\end{defi}

Las condiciones anteriores pueden resumirse diciendo que para todo subconjunto de $r$ hiperplanos del arreglo, con $0\leq r\leq d$, la intersecci\'on es un subespacio af\'{\i}n de dimensi\'on $d-r$; y que para m\'as de $d$ hiperplanos, la intersecci\'on siempre es vac\'{\i}a. Podemos ver la intersecci\'on de un conjunto de $r$ hiperplanos como el conjunto soluci\'on de un sistema de $r$ ecuaciones con $d$ inc\'ognitas. Un arreglo es gen\'erico cuando 
cualesquiera $r$ ecuaciones son linealmente independientes.

Luego de precisar estos conceptos, es posible demostrar la generalizaci\'on natural de los Teoremas \ref{2D} y \ref{3D} en cualquier dimensi\'on. Sin embargo, procederemos de manera diferente; vamos a 
desarrollar herramientas generales que nos permitan entender esta situaci\'on m\'as conceptualmente.


\section{\textsf{Posets de intersecci\'on, funciones de M\"obius, y polinomios \\ caracter\'{\i}sticos}}


Recordemos que un \emph{conjunto parcialmente ordenado} o \emph{poset}\footnote{del ingl\'es ``partially ordered set"} $P$ es un conjunto $P$ junto con una relaci\'on binaria $\leq$ de ``orden parcial" \ tal que:
\begin{itemize}
\item Para todo $x$ se tiene que $x \leq x$.
\item Si $x \leq y$ y $y \leq x$ entonces $x=y$.
\item Si $x \leq y$ y $y \leq z$ entonces $x \leq z$.
\end{itemize}

\begin{defi}
El \emph{poset de intersecci\'on} $L(\A)$  de $\A$ es el conjunto de las diferentes intersecciones no vac\'{\i}as de los subconjuntos de $\A$, incluyendo a los mismos hiperplanos y al espacio completo (que es la intersecci\'on del subconjunto vac\'{\i}o). Este conjunto est\'a parcialmente ordenado por la relaci\'on $X\leq Y$ para cada par de elementos $X,Y\in L(\A)$ tales que $Y\subseteq X$.
\end{defi}


\begin{ej} \label{ejposet}
La Figura \ref{poset} muestra un arreglo en $\R^2$ y su poset de intersecci\'on.
\begin{figure}[htb]\centering
\includegraphics[scale=0.7]{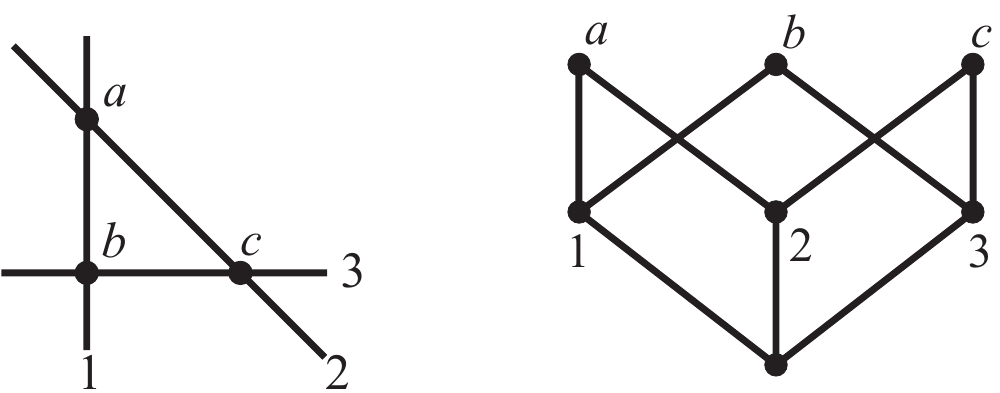}
\caption{Poset de intersecci\'on de un arreglo.}\label{poset}
\end{figure}
\end{ej}


\begin{defi}
Sea $P$ un poset finito con elemento m\'{\i}nimo $\hat0$. La \emph{funci\'on de M\"obius} $\mu:P\rightarrow \Z$ de $P$ se define recursivamente por: 
\begin{enumerate}[i)]
\item $\mu(\hat0)=1$.
\item $\mu(X)=-\sum_{Y< X}\mu(Y)$, para todo $X\neq \hat0$.
\end{enumerate}
 \end{defi}

Gian-Carlo Rota \cite{R} fue el precursor  del estudio de la funci\'on de M\"obius de un poset, una teor\'{\i}a muy elegante que conecta resultados en la teor\'{\i}a de los n\'umeros (la funci\'on de M\"obius cl\'asica), la combinatoria enumerativa (la f\'ormula de inclusi\'on-exclusi\'on), y la topolog\'{\i}a (la caracter\'{\i}stica de Euler), entre otras. Nosotros aplicaremos esta teor\'{i}a a los arreglos de hiperplanos.
%


\begin{defi}
El \emph{polinomio caracter\'{\i}stico} $\chi_{\A}(t)$ del arreglo $\A$ es
$$\chi_{\A}(t)=\sum_{X\in L(\A)} \mu(X)t^{\dim(X)}.$$
\end{defi}

Vale la pena tener en cuenta que la funci\'on de M\"obius y el polinomio caracter\'{\i}stico de un arreglo $\A$ dependen \'unicamente del poset $L(\A)$ y la dimensi\'on del espacio.

\begin{ej}\label{figu4}
La Figura \ref{muposet} muestra los valores de la funci\'on $\mu$ en el poset $L(\A)$ del ejemplo \ref{figu4}. El polinomio caracter\'{\i}stico de $\A$ es $\chi_{\A}(t)=t^2-3t+3.$
\end{ej}
\begin{figure}[h]\centering
\includegraphics[scale=0.7]{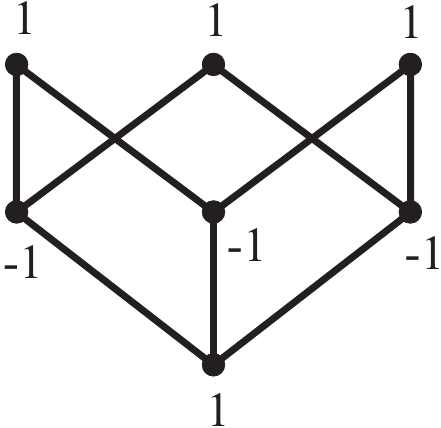}
\caption{Funci\'on de M\"obius de $L(\A)$.}\label{muposet}
\end{figure}

\begin{teor}\label{generico}
El polinomio caracter\'{\i}stico de un arreglo de $n$ hiperplanos en posici\'on general en $\R^d$ es \begin{equation}\label{pcpg}
\chi(t)=t^d-{n\choose 1}t^{d-1}+{n\choose 2}t^{d-2}-\cdots+(-1)^d{n\choose d}.
\end{equation}
\end{teor}

\begin{proof} 
En este caso los elementos de $L(\A)$ de dimensi\'on $d-k$ son todas las posibles intersecciones de $k$ hiperplanos, que son todas distintas; hay ${n\choose k}$ de ellas. Adem\'as, dado un elemento $X$ de $L(\A)$ de dimensi\'on $d-k$ que es la intersecci\'on de $\B \subseteq \A$ con $|\B|=k$, los elementos $Y$ tales que $Y \leq X$ son las intersecciones de los subconjuntos de $\B$. Por lo tanto hay ${k\choose l}$ elementos de $L(\A)$ de dimensi\'on $d-l$ menores que $X$. Usando esto, demostraremos inductivamente que $\mu(X)=(-1)^k$.

Para ver esto, primero observamos que $\mu(\R^d)=(-1)^0=1$. Para otro elemento $X\in L(\A)$ de dimensi\'on $d-k$, se encuentra que 
$$\mu(X)=-\sum_{Y<X} \mu(Y)=-\left(1-{k\choose 1}+{k\choose 2}-\cdots+(-1)^{k-1}{k\choose k-1}\right)=(-1)^k,$$
ya que la suma contiene $k \choose l$ elementos de dimensi\'on $d-l$, cuya funci\'on de M\"obius es $(-1)^l$ por inducci\'on. 
El resultado se sigue. 
\end{proof}

Existe una clara relaci\'on entre el polinomio caracter\'{\i}stico de un arreglo gen\'erico (Teorema \ref{generico}) y los n\'umeros de regiones y de regiones acotadas de tal arreglo para $d=2,3$ (Teoremas \ref{2D} y  \ref{3D}). Nuestro siguiente objetivo es aclarar esta relaci\'on, en el Teorema \ref{zas}.

\begin{defi} La \emph{dimensi\'on} de un arreglo $\A$ en $\R^n$ es $\dim(\A)=n$. El \emph{rango} de $\A$ es la dimensi\'on del espacio $V_{\A}$ generado por los vectores $v_i$ normales a cada uno de los hiperplanos, y se denota $\rank(\A)$. Decimos que $\A$ es \emph{esencial} si $\dim(\A) = \rank({\A})$.\end{defi}

Si un arreglo $\A$ no es esencial, la intersecci\'on de los hiperplanos de $\A$ es un espacio af\'{\i}n $L$. Podemos \emph{esencializar} el arreglo, considerando el arreglo $\A^{es} = \{H\cap L^{\perp} \, : \, H \in \A\}$ de hiperplanos en el espacio $L^{\perp}$ ortogonal a $L$. Es claro que $\A$ y $\A^{es}$ tienen el mismo poset de intersecci\'on, y tambi\'en que las regiones de estos dos arreglos est\'an en biyecci\'on.

\begin{defi}
Una regi\'on de $\A$ es \emph{relativamente acotada} si la regi\'on correspondiente de $\A^{es}$ es un conjunto acotado (en el sentido usual). 
 Sean $\reg(\A)$ y $\ac(\A)$ el n\'umero de regiones de $\A$ y el n\'umero de regiones relativamente acotadas de $\A$, respectivamente.
  \end{defi}

La definici\'on anterior es necesaria ya que un arreglo no esencial no tiene regiones acotadas.
%
%
%
Para simplificar, de ahora en adelante, cuando hablemos de regiones acotadas estaremos refiri\'endonos a las regi\'ones relativamente acotadas. 
El siguiente teorema relaciona el n\'umero de regiones y de regiones acotadas de un arreglo con su polinomio caracter\'{\i}stico.

\begin{teor}[Zaslavsky, \cite{zas}]\label{zas}
Los n\'umeros de regiones y de regiones acotadas de un arreglo  $\A$ en $\R^d$ est\'an dadas por:
\[
\reg(\A) =  (-1)^{d}\chi_{\A}(-1), \qquad \qquad \ac(\A) = (-1)^{\rank(\A)}\chi_{\A}(1).
\]
\end{teor}

Este teorema se demostrar\'a usando una idea que ya hemos expuesto en los ejemplos iniciales, que se conoce como el \emph{m\'etodo de eliminaci\'on/contracci\'on}. Esta es una de las t\'ecnicas m\'as \'utiles para llevar a cabo argumentos inductivos en arreglos de hiperplanos. Consiste en expresar cierta informaci\'on de un arreglo en t\'erminos de los siguientes dos arreglos m\'as peque\~nos:

\begin{defi}
Sea $\A$ un arreglo de hiperplanos en $\R^d$ y $H$ un hiperplano de $\A$.
\begin{itemize}
\item La \emph{eliminaci\'on de $H$} es el arreglo $\A -H$ en $\R^d$ formado por todos los hiperplanos de $\A$ excepto $H$.
\item La \emph{contracci\'on de $H$} es el arreglo $\A /H=\{H'\cap H:H'\in \A-H\}$ de hiperplanos en $H \cong \R^{d-1}$.
 \end{itemize}
\end{defi}

\begin{proof}[Demostraci\'on del Teorema \ref{zas}]
Usando el mismo argumento de los Teoremas \ref{2D} y \ref{3D} obtenemos las f\'ormulas recursivas
\begin{equation}\label{regdr}
\reg(\A)= \reg(\A-H)+\reg(\A/H), \qquad \qquad \ac(\A)=\ac(\A-H)+\ac(\A/H).
\end{equation} 
Basta ver que cada regi\'on de $\A/H$ est\'a partiendo una regi\'on de $\A-H$ en dos, dando como resultado todas las regi\'ones de $\A$. Algo similar sucede con las regiones acotadas.

Por otro lado, vamos a demostrar la siguiente f\'ormula recursiva para el polinomio carac\-ter\'{\i}stico de un arreglo
\begin{equation}\label{dr}
\chi_{\A}(q)=\chi_{\A-H}(q)-\chi_{\A/H}(q).
\end{equation}

Para hacerlo, vamos a demostrar que 
\begin{equation}\label{recx}
\mu_{\A}(X)=\mu_{\A-H}(X)-\mu_{\A/H}(X),
\end{equation}
para todo $X\in L(\A)$, donde el sub\'{\i}ndice indica el arreglo en el cual se calcula la funci\'on. Si $X \notin L(\A')$ definimos $\mu_{\A'}(X)=0$.

Esto lo podemos ver por inducci\'on en $d-l$ donde $l$ es la dimensi\'on de $X$. Es decir, por inducci\'on comenzando por los menores elementos de $L(\A)$. Para $X=\R^d$, se cumple que $X\notin \A/H$, por lo que $\mu_{\A}(\R^d)=\mu_{\A-H}(\R^d)-\mu_{\A/H}(\R^d)=1$. Para los hiperplanos del arreglo tambi\'en se verifica la identidad anterior, notando que si $X\neq H$, entonces $\mu_{\A/H}(X)=0$, mientras que si $X=H$, entonces $\mu_{\A-H}(X)=0$ y  $\mu_{\A/H}(H)=1$.
Para los dem\'as elementos $X\in L(\A),$ se tiene por hip\'otesis de inducci\'on que
\begin{eqnarray*}
\mu_{\A}(X)&=&-\sum_{Y<_\A X}\mu_{\A}(Y)=-\sum_{Y<_\A X}\mu_{\A-H}(Y)+\sum_{Y<_\A X}\mu_{\A/H}(Y)\\
&=&-\sum_{Y<_{\A-H} X}\mu_{\A-H}(Y)+\sum_{Y< _{\A/H} X}\mu_{\A/H}(Y) =\mu_{\A-H}(X)-\mu_{\A/H}(X)
\end{eqnarray*}
donde $Y <_\B X$ denota que $Y<X$ en el poset de intersecci\'on del arreglo $\B$. El resultado se sigue.
\end{proof}

Como corolario del teorema anterior, se encuentra que el n\'umero de regiones y el n\'umero de regiones acotadas de un arreglo tan s\'olo dependen de su poset de intersecci\'on. Veremos varias aplicaciones de este teorema. La primera es la generalizaci\'on de  los Teoremas \ref{2D} y \ref{3D} a $d$ dimensiones.

\begin{teor}\label{nD}
Cualquier arreglo de $n$ hiperplanos en posici\'on general en $\R^d$ tiene el m'aximo n'umero de regiones $\reg_d(n)$ y de regiones acotadas $\ac_d(n)$ entre todos los arreglos de $n$ hiperplanos en $\R^d$. Estos n\'umeros son
\begin{eqnarray*}
\reg_d(n) &=& {n \choose d} + {n \choose d-1} + \cdots +  {n \choose 1} + {n \choose 0}, \\
\ac_d(n) &=& {n \choose d} - {n \choose d-1} + \cdots +(-1)^{d-1} {n \choose 1} +(-1)^d {n \choose 0}.
\end{eqnarray*}
\end{teor}

\begin{proof}
Esta es una consecuencia inmediata de los Teoremas \ref{generico} y  \ref{zas}.
\end{proof}

\section{\textsf{Arreglos en campos finitos}}

Vamos a considerar ahora arreglos de hiperplanos sobre otros campos distintos a $\R$. En particular, resulta \'util considerar arreglos en un espacio vectorial finito $\F_q^d$, a\'un si nuestro inter\'es principal son los arreglos reales. Ac\'a $\F_q$ es el campo finito de $q$ elementos, donde $q=p^\alpha$,  $p$ es un n\'umero primo y  $\alpha \in \Z_{>0}$.
%
%
Algunas nociones como el n\'umero de regiones no existen (pues $\F_q$ no es un campo ordenado), pero otras nociones  como el poset de intersecci\'on permanecen iguales, y permiten darle nuevas interpretaciones al polinomio caracter\'{\i}stico.

Supongamos que $\A$ es un arreglo en $\R^d$ cuyas ecuaciones tienen coeficientes enteros. Utilizando las mismas ecuaciones sobre el campo finito $\F_q$, obtenemos un arreglo $\A_q$ en $\F_q^d$. El siguiente resultado reduce el c\'alculo del polinomio caracter\'{\i}stico de $\A$ a un problema enumerativo en $\F_q^d$.

\begin{teor}[Crapo-Rota, Athanasiadis]\label{fp}
El n\'umero de puntos en $\F_q^d$ que no pertenecen al arreglo $\A_q$ es igual a $\chi_{\A}(q)$ para todo $q=p^\alpha$ donde $p$ es un primo suficientemente grande.
\end{teor}

\begin{proof}
Cada intersecci\'on de hiperplanos $Y\in L(\A_q)$ es un espacio af\'{\i}n de dimensi\'on $k$, y por lo tanto tiene $q^k$ elementos. As\'{\i} cada sumando $\mu(Y)q^k$ en la definici\'on de $\chi_{\A}(q)$ cuenta el n\'umero de puntos de $Y$ multiplicado por $\mu(Y)$. 
Luego, cada punto $v$ de $\F_q^d$ contribuye a $\chi_{\A}(q)$ un total de $\sum_{Y \, : \, v \in Y}\mu(Y) = \sum_{Y\leq X}\mu(Y)$, donde $X$ es el m\'{\i}nimo elemento de $L(\A)$ que contiene a $v$. Esta suma es igual a $0$ si $X \neq \F_q^d\in L(\A_q)$ y es igual a $1$ si $X = \F_q^d$ es el m'inimo elemento del poset. Por lo tanto los \'unicos puntos que contribuyen son los que no est\'an en ning\'un hiperplano de $\A$, y $\chi_{\A_q}(q) = | \F_q^d - \A_q|$.


Por \'ultimo, observamos que cada elemento de $L_{\A}$ es el espacio de soluciones a un sistema de ecuaciones lineales, y si $p$ es suficientemente grande se tiene que hay un elemento correspondiente de $L_{\A_q}$ que satisface las mismas ecuaciones (en $\F_q$). En ese caso, tendremos que $L_{\A} \cong L_{\A_q}$ y tambi\'en $\chi_{\A} = \chi_{\A_q}$. El resultado se sigue.
\end{proof}

El teorema anterior puede ser de gran ayuda para calcular el polinomio caracter\'{\i}stico de un arreglo, pues en muchos casos es f\'acil contar los puntos en cuesti\'on en t\'erminos de $q$. Con frecuencia nos enfocaremos en el caso en que $q$ es un primo suficientemente grande.


\begin{ej}
Volviendo al arreglo del Ejemplo \ref{ejposet}, consideremos el arreglo en $\F_q^2$ dado por las ecuaciones $$x=0,\quad y=0,\quad  x+y=1.$$ Es f\'acil ver que el n\'umero de puntos que no pertenece a ning\'un hiperplano est\'a dado por $$(q-1)^2-(q-2)=q^2-3q+3=\chi_{\A}(q).$$
\end{ej}

En este caso, el m\'etodo de los campos finitos nos permite verificar f\'acilmente que el polinomio caracter\'{\i}stico coincide con el del Ejemplo \ref{figu4}. Para apreciar verdaderamente su utilidad, es necesario estudiar ejemplos m\'as interesantes, en los cuales el m\'etodo de campos finitos puede simplificar incre\'{\i}blemente los c\'alculos. En la secci\'on siguiente vamos a ver varios ejemplos importantes.

\section{\textsf{Varios arreglos interesantes}}

\subsection{\textsf{El arreglo de trenzas}}

\begin{defi}
El \emph{arreglo de trenzas} (tambi\'en conocido como \emph{arreglo trenza}), denotado por $\mathcal{B}_n$, es un arreglo en $\R^n$, formado por todos los ${n \choose 2}$ planos $H_{ij}$ de la forma $x_i-x_j=0$, para $1\leq i<j\leq n$.
\end{defi}

Este arreglo puede verse como la variedad algebraica $Z(\Delta _n)$, que consiste de los ceros del polinomio de Vandermonde $\prod_{1 \leq i < j \leq n} (x_i-x_j)$. Veamos c\'omo el m\'etodo de campos finitos nos permite calcular f\'acilmente el polinomio caracter\'{\i}stico de este arreglo.

\begin{teor}\label{pcbraid}
El polinomio caracter\'{\i}stico del arreglo trenza $\B_n$ es
$$\chi_{\B_n}(t) = t(t-1)(t-2)(t-3)\cdots (t-n+1).$$
\end{teor}
\begin{proof}
Al considerar $\B_n$ sobre el campo finito $\F_q$  para un primo $q$ suficientemente grande, vemos que el n\'umero de puntos de $\F_q^n$ que no pertenece a ning\'un hiperplano est\'a dado por $$q(q-1)(q-2)(q-3)\cdots (q-n+1)=\chi_{\B_n}(q),$$ pues tan solo debemos escoger como coordenadas a $n$ elementos distintos de $\F_q$. Para la primera coordenada hay $q$ posibilidades, para la siguiente $q-1$ y as\'{\i} sucesivamente. Como esta igualdad es v\'alida para infinitos valores de $q$, debe ser una igualdad de polinomios.
\end{proof}

Para el arreglo de trenzas tambi\'en es posible, pero mucho m\'as dif\'{\i}cil, calcular el polinomio caracter\'{\i}stico usando la funci\'on de M\"obius. Por ejemplo, para $n=4$,  el poset de intersecci\'on de $\B_4$ se muestra en la Figura \ref{posetp4}.

\begin{figure}[htb]\centering
\includegraphics[width=7cm]{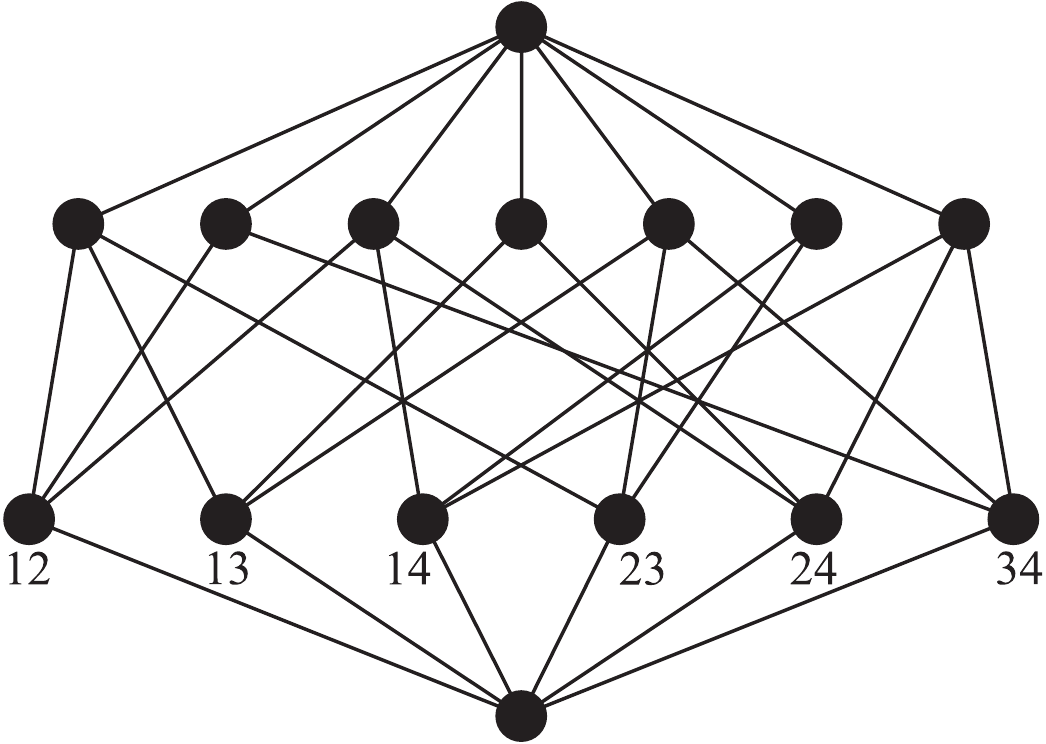}
\caption{Poset de intersecci\'on de $\B_4$.}\label{posetp4}
\end{figure}

Cada elemento marcado con un par $ij$ representa al hiperplano $x_i=x_j$.

\begin{cor}
El arreglo de trenzas $\B_n$ tiene $\reg(\B_n)=n!$ regiones, ninguna de las cuales es acotada.
\end{cor}

\begin{proof}
Esta es una consecuencia inmediata del Teorema de Zaslavsky (Teorema \ref{zas}) y del Teorema \ref{pcbraid}. Es f\'acil e ilustrativo dar una prueba directa de estas afirmaciones. La segunda igualdad es clara ya que todos los hiperplanos pasan por el origen. La primera igualdad tiene una sencilla explicaci\'on combinatoria. Cada regi\'on est\'a dada por un sistema de desigualdades, donde para cada par $i,j$ se elije si $x_i<x_j$ o $x_j<x_i$. Juntando todas las desigualdades, podemos encontrar en qu\'e orden se encuentran todos los $x_i$. Adem\'as, cada orden posible de las $n$ coordenadas determina una \'unica regi\'on. As\'{\i} concluimos que las regiones de $\B_n$ est\'an en biyecci\'on con las $n!$ permutaciones de $\{1, \ldots, n\}$. 
\end{proof}

\begin{ej}
  En la Figura \ref{braid2} se muestran las regiones del arreglo de trenzas para el caso $n=3$.

  \begin{figure}[htb]\centering
    \includegraphics[width=7cm]{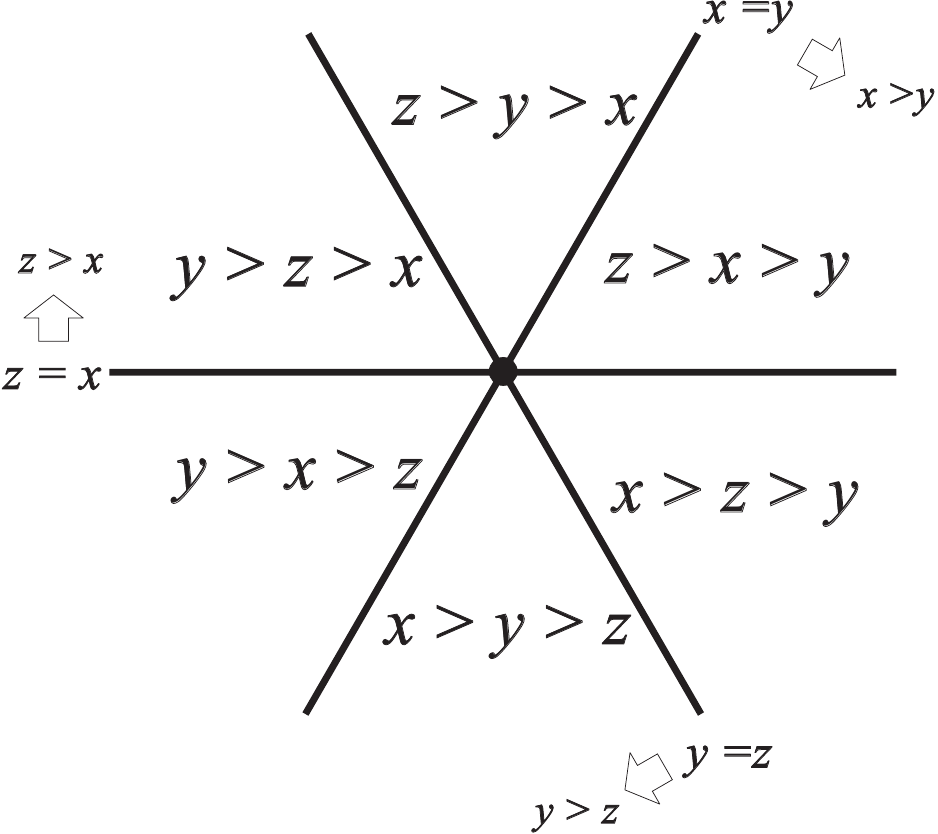}
    \caption{Regiones del arreglo $\B_3$.}\label{braid2}
  \end{figure}

  Este es un arreglo en $\R^3;$ pero como todos los hiperplanos contienen la l\'{\i}nea $\R(1,1,1)$,  dibujamos su esencializaci\'on, intersectando el arreglo con el plano $x+y+z=0$. 
\end{ej}

\subsection{\textsf{Arreglos Gr\'aficos}}

Un tema cl\'asico de la teor\'{\i}a de grafos es el de las  \emph{coloraciones propias} de un grafo $G$. Estas son las coloraciones de los v\'ertices del grafo tales que dos v\'ertices unidos por un arco no pueden tener el mismo color.  Se puede demostrar que existe un polinomio $\chi_G(x)$, conocido como el \emph{polinomio crom\'atico} de $G$, con la siguiente propiedad:
si tenemos $t$ colores disponibles, 
 el n\'umero de coloraciones propias del grafo $G$ es igual a $\chi_G(t)$. Algunas propiedades de estos polinomios se pueden encontrar en \cite{S1}.

\begin{defi}
Sea $G$ un grafo con $n$ v\'ertices numerados de $1$ a $n$. El \emph{arreglo gr\'afico} de hiperplanos $\A_G$ en $\R^n$ est\'a dado por los hiperplanos de la forma $x_i-x_j=0$ para cada par de v\'ertices $i,j$ que est\'en unidos por un arco. 
\end{defi}

\begin{teor}
Para todo grafo $G$, el polinomio crom\'atico de $G$ es igual al polinomio caracter\'{\i}stico de $\A_G$: $\chi_G(t)=\chi_{\A_G}(t)$.
\end{teor}
\begin{proof}
Considerando el arreglo $\A_G$ sobre el campo finito $\F_q$, vemos que los puntos de $\F_q^n$ que no pertenecen a ninguno de los hiperplanos est\'an en biyecci\'on con las coloraciones propias de $G$ con $q$ colores: la $i$-\'esima coordenada del punto nos da el color del v\'ertice $i$ en la coloraci\'on. Usando el Teorema \ref{fp} concluimos que $\chi_G(q)=\chi_{\A_G}(q)$ para casi todo primo $q$, y por lo tanto $\chi_G=\chi_{\A_G}$ como polinomios. \end{proof}


Tambi\'en es posible dar una interpretaci\'on al n\'umero de regiones un arreglo gr\'afico, con lo cual el polinomio caracter\'{\i}stico contendr\'{\i}a a\'un m\'as informaci\'on del grafo. Para ver esto, n\'otese que cada regi\'on de un arreglo gr\'afico $\A_G$ est\'a determinada por un sistema de desigualdades de la forma $x_i<x_j$, para cada pareja de v\'ertices $i,j$ unida por un arco.
Estas desigualdades se pueden marcar en el grafo, poniendo en cada arco una flecha de $i$ hasta $j$ si $x_i<x_j$ o viceversa.
As\'{\i}, cada sistema de desigualdades define una orientaci\'on de los arcos de $G$, como se muestra en la figura \ref{oa}.
\begin{figure}[htb]\centering
\includegraphics[scale=0.7]{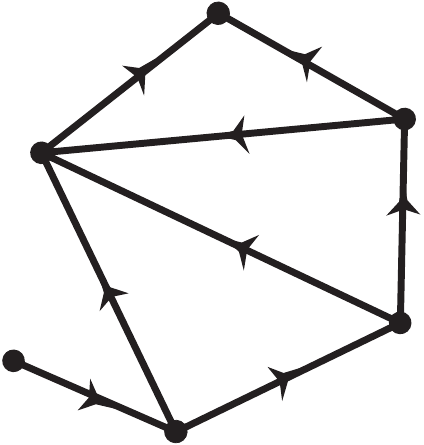}
\caption{Orientaci\'on ac\'{\i}clica de un grafo.}\label{oa}
\end{figure}

\begin{lem}
Un sistema de desigualdades para un arreglo gr\'afico tiene soluci\'on si y s\'olo si su orientaci\'on correspondiente no tiene ciclos.
\end{lem}

\begin{proof}
Si el sistema de desigualdades tiene soluci\'on, el grafo no puede tener ciclos, pues esto implicar\'{\i}a desigualdades de la forma $x_i<x_j<\cdots<x_i$. Por otro lado, dada una orientaci\'on ac\'{\i}clica, vamos a mostrar que el sistema de desigualdades correspondiente tiene  soluci\'on, de forma inductiva sobre el n\'umero de v\'ertices $n$. Para $n=1$, el resultado es obvio. Para $n \geq 2$,  como el grafo es ac\'{\i}clico, debe haber al menos un v\'ertice $v$ del que no sale ning\'una flecha. Podemos entonces retirar a $v$ del grafo, junto con todos los arcos que llegan a $v$. El grafo $G'$ orientado que resulta tambi\'en es ac\'{\i}clico y con menos v\'ertices que el original. Usando la hip\'otesis de inducci\'on, es posible construir una soluci\'on al sistema correspondiente a $G'$; y d\'andole a la variable asociada  a $v$ un valor menor a todos los dem\'as, obtenemos una soluci\'on a nuestro sistema de desigualdades.
\end{proof}

\begin{teor}[Stanley]
El n\'umero de orientaciones ac\'{\i}clicas de un grafo $G$ es $(-1)^n\chi_G(-1)$.
\end{teor}

\begin{proof}
Basta notar que 
$$\reg(\A_G)=(-1)^n\chi_{\A_G}(-1)=(-1)^n\chi_G(-1).$$
y que las orientaciones ac\'{\i}clicas de $G$ est\'an en biyecci\'on con las regiones del arreglo $\A_G$. 
\end{proof}

Es muy interesante que el estudio de arreglos de hiperplanos nos haya permitido demostrar f\'acilmente este teorema puramente combinatorio.

Volvamos brevemente al tema de la ubicaci\'on de los ceros de polinomios combinatorios, que ya apareci\'o en el primero de estos tres art\'{\i}culos. El problema de la ubicaci\'on de los ceros de $\chi_G(t)$ ha recibido gran inter\'es, en parte gracias a su relaci\'on con el famoso Teorema de los Cuatro Colores. Este teorema dice que cualquier mapa puede ser coloreado con cuatro colores sin que haya dos pa\'{\i}ses vecinos del mismo color. Esta afirmaci\'on es equivalente a decir que si $G$ es plano (es decir, que puede pintarse en el plano sin que los arcos se corten), entonces $\chi_G(4)\neq 0.$ Otros resultados conocidos son los siguientes \cite{So, So2}:

\begin{itemize}
\item Todas las ra\'{\i}ces reales de $\chi_G$ son mayores o iguales que 0.
\item Para un grafo plano, todas las ra\'{\i}ces reales son menores que 5.  (Dijimos tambi\'en que no pod\'{\i}an ser 4, pero no se sabe si puede haber ra\'{\i}ces entre 4 y 5).
\item Cualquier ra\'{\i}z (real o compleja) $r$ satisface que $|r|<8d$, donde $d$ es el m\'aximo grado de un v\'ertice del grafo. 
\item El conjunto de todas las ra\'{\i}ces de todos los polinomios crom\'aticos es denso en el plano complejo $\Cx$.
\end{itemize}


\subsection{\textsf{Arreglo de Catalan}}

Los \emph{n\'umeros de Catalan} $C_n$ est\'an dados por la f\'ormula $$C_n=\frac1{n+1}{2n \choose n},$$ para $n\in\N$.
Los \emph{n\'umeros de Catalan} aparecen naturalmente en una gran cantidad de contextos matem\'aticos. En particular, uno de los ejercicios del libro \cite{Sv2} contiene cientos de problemas combinatorios cuya respuesta son los n\'umeros de Catalan.
Un ejemplo importante es el siguiente: el n\'umero de Catalan $C_n$ es el n\'umero de \emph{sucesiones de votaci\'on} $b_1,b_2,\ldots,b_{2n}$ donde cada $b_i$ es $1$ o $-1$, tales que
\[
b_1+\cdots+b_i \geq 0  \quad\textrm{ para } i=1, \ldots, 2n-1 
\qquad \quad \textrm{y } \qquad \quad
b_1+\cdots+b_{2n}= 0.
\]
El nombre proviene de un modelo de unas elecciones donde $2n$ votantes votan por uno de dos candidatos A y B. Los votos se reciben en orden, y el candidato A nunca est\'a detr\'as del candidato B, pero al final el resultado es un empate.

Los n\'umeros de Catalan cumplen la relaci\'on de recurrencia: 
\[
C_0=1, \qquad  \qquad C_{n+1}=C_0C_n+C_1C_{n-1}+\cdots+C_nC_0 \quad {(n \geq 0)}.
\]


\begin{defi}
El  \emph{arreglo de Catalan} $\C_n$ es el arreglo en $\R^n$ formado por los hiperplanos $$x_i-x_j=-1,\qquad x_i-x_j= 0, \qquad x_i-x_j= 1 \qquad \qquad (1\leq i<j\leq n).$$
\end{defi}

\begin{teor}\label{pccat}
El polinomio caracter\'{\i}stico del arreglo de Catalan es $$\chi_{\C_n}(t)=t(t-n-1)(t-n-2)(t-n-3)\cdots (t-2n+1).$$
\end{teor}

\begin{proof}
Usaremos nuevamente el m\'etodo de campos finitos (Teorema \ref{fp}). Sea $q>2n$ un primo, y encontremos el n\'umero de formas de seleccionar $n$ valores $(x_1,x_2,\ldots,x_n)$ en $\F_q^n$, de forma que no haya valores repetidos ni dos valores consecutivos entre las coordenadas. Para el valor de $x_1$ hay $q$ posibilidades. Para escoger los valores que pueden tomar las otras coordenadas $x_i$, hay en total ${q-n-1\choose n-1}(n-1)!$ posibilidades; ya que si escogemos $n-1$ enteros $z_1<z_2<\cdots<z_{n-1}$ entre $1$ y $q-n-1$, podemos darle a las otras $n-1$ coordenadas de $x$ los valores $x_1+z_1+1, x_1+z_2+2, \ldots,$ y $x_1+z_{n-1}+{n-1}$ en cualquiera de los $(n-1)!$ ordenes posibles. Los $n$ valores que resultan son distintos y no hay dos consecutivos. En la direcci\'on contraria, dado un punto $x$ en $\F_q^n$ que no est\'a en ning\'un hiperplano, es f\'acil recuperar los valores de $x_1, z_1, \ldots, z_{n-1}$. 

Concluimos entonces que el n\'umero de puntos de $\F_q^n$ que no est\'an en ninguno de los hiperplanos de $\C_n$ es $$q{q-n-1\choose n-1}(n-1)!=q(q-n-1)(q-n-2)(q-n-3)\cdots (q-2n+1),$$ y de aqu\'{\i} el resultado se sigue.
\end{proof}

\begin{teor} El n\'umero de regiones determinadas por el arreglo $\C_n$ es  $n! C_n$ y el n\'umero de regiones acotadas es $n!C_{n-1}$.
\end{teor}

\begin{proof}
Combinando el Teorema de Zaslavsky (Teorema \ref{zas}) con el Teorema \ref{pccat}, obtenemos una prueba directa de estos resultados. Vamos a esbozar una segunda demostraci\'on que aclara la relaci\'on entre los arreglos de Catalan y los n\'umeros de Catalan. La demostraci\'on completa se encuentra, por ejemplo, en \cite{tg}.


El arreglo de Catalan contiene al arreglo de trenzas, que divide a $\R^n$ en $n!$ sectores iguales. En cada sector, el orden de las coordenadas de los puntos est\'a dado por una permutaci\'on fija. Consideremos por ejemplo la region $x_1 > x_2 > \cdots > x_n$ de $\B_n$; vamos a demostrar que el arreglo de Catalan la divide en $C_n$ subregiones. Para especificar una de estas subregiones, debemos decidir si $x_i-x_j < 1$ o $x_i - x_j > 1$ para cada $i<j$. (Ya sabemos que $x_i-x_j > -1$ para cada $i<j$.) En otras palabras, debemos decidir el orden de los n\'umeros $x_1, \ldots, x_n, x_1+1, \ldots, x_n+1$, sabiendo que $x_1> \cdots >x_n$ y $x_1+1 > \cdots > x_n+1$.  Para cada orden posible, reemplacemos cada $x_i$ por un $-1$ y cada $x_i+1$ por un $1$. Por ejemplo, el orden $x_1+1 > x_2+1 > x_1 > x_3+1 > x_2 > x_4+1 > x_3 > x_4$ se convierte en la sucesi\'on $(1,1,-1,1,-1,1,-1,-1)$. Es claro que el resultado es una sucesi\'on de votaci\'on, y cada sucesi\'on de votaci\'on corresponde a una subregi\'on. Adem\'as, uno puede verificar que la regi\'on es acotada cuando todas las sumas parciales  $b_1+\cdots + b_j$ con $1 \leq j \leq 2n-1$ son positivas. Tambi\'en es f\'acil ver  que hay exactamente $C_{n-1}$ sucesiones con esa propiedad. Por lo tanto la regi\'on $x_1 > x_2 > \cdots > x_n$ de $\B_n$ contiene $C_n$ regiones de $\C_n$, de las cuales $C_{n-1}$ son acotadas. El resultado se sigue. 
%
\end{proof}

\subsection{\textsf{Arreglo de Shi}}


\begin{defi} 
El arreglo de Shi $\Shi_n$ en $\R^n$  est\'a formado por los hiperplanos
 $$x_i-x_j= 0, \qquad x_i-x_j= 1 \qquad \qquad (1\leq i< j\leq n).$$
\end{defi}

\begin{prop}\label{pcshi}
El polinomio caracter\'{\i}stico para $\Shi_n$ es $\chi_{\Shi_n}(t)=t(t-n)^{n-1}$.
\end{prop}

\begin{proof}
Nuevamente vamos a usar el m\'etodo de campos finitos; contemos los puntos de $\F_q^n$ que no satisfacen ninguna de las ecuaciones del arreglo de Shi. Representemos a cada uno de esos puntos $x=(x_1, \ldots, x_n)$ por una $q$-tupla $x^{-1}=(y_0, \ldots, y_{q-1})$ de n\'umeros y s\'{\i}mbolos $\circ$, donde $y_i=\circ$ si no existe ning\'un $j \in [n]$ tal que $x_j=i$, y $y_i = j$ si $j$ es el elemento de $[n]$ tal que $x_j=i$. Tal elemento debe ser \'unico ya que $x$ no tiene valores repetidos. Por ejemplo, al punto $x=(3,2,8,1,5,7,12,11) \in \F_{13}^8$ le corresponde la $13$-tupla $x^{-1} = (\circ, 4, 2, 1, \circ, 5, \circ, 6, 3, \circ, \circ, 8, 7)$. Observemos que cada sucesi\'on de n\'umeros entre dos $\circ$ consecutivos es decreciente, ya que $x_i-x_j \neq 1$ para $i<j$.

Podemos construir estas $q$-tuplas de una manera alternativa que nos permitir\'a contarlas f\'acilmente. Para hacerlo, pensamos en la $q$-tupla $x^{-1}$ como un vector escrito alrededor de un c\'{\i}rculo, reflejando la estructura c\'{\i}clica del campo $\F_q$ bajo adici\'on. Primero dibujamos $q-n$ s\'{\i}mbolos indistinguibles $\circ$. Luego ubicamos un $1$ entre cualesquiera dos de ellos, teniendo en cuenta que los s\'{\i}mbolos $\circ$ son indistinguibles por el momento. Luego ubicamos cada uno de los n\'umeros $2, \ldots, n$ en alguno de los $q-n$ espacios entre dos $\circ$ consecutivos; esto lo podemos hacer de $(q-n)^{n-1}$ maneras.
Si un espacio entre dos $\circ$ consecutivos contiene varios n\'umeros, los ubicamos en orden decreciente en el sentido de las manecillas del reloj. Esto determina el orden c\'{\i}clico de los s\'{\i}mbolos de la $q$-tupla. Por \'ultimo, para determinar la $q$-tupla exactamente, necesitamos pasar del orden c\'{\i}clico a un orden lineal, eligiendo la posici\'on del $1$ en la $q$-tupla; es decir, el valor de $x_1$. Esto nos da un total de $q(q-n)^{n-1}$ posibilidades. Es f\'acil verificar que este procedimiento produce precisamente los puntos de $\F_q^n$ que no est\'an en ning\'un hiperplano del arreglo de Shi $\Shi_n$.
\end{proof}

\begin{cor}\label{rshi}
El arreglo $\Shi_n$ tiene $(n+1)^{n-1}$ regiones, de las cuales $(n-1)^{n-1}$ son acotadas.
\end{cor}

\begin{proof}
Esta es una consecuencia inmediata de la Proposici\'on \ref{pcshi} y el Teorema de Zaslavsky.
\end{proof}

El arreglo de Shi est\'a cercanamente relacionado a ciertas listas de n\'umeros conocidas como \emph{funciones de parqueo}. Para explicar su definici\'on, consideremos la siguiente situaci\'on:

\begin{figure}[htb]\centering
\includegraphics[width=12cm]{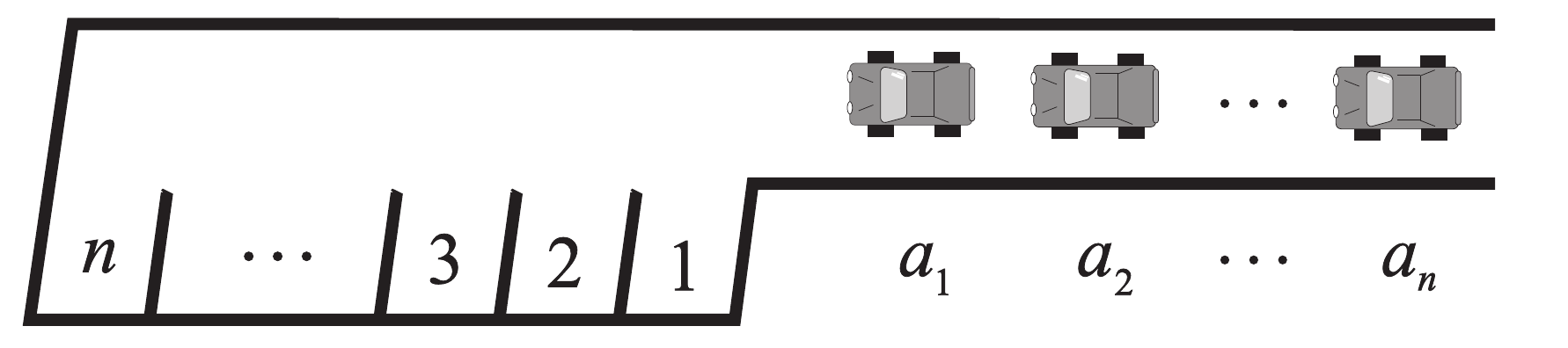}
\caption{Funciones de parqueo.}\label{pkf}
\end{figure}

En un parqueadero se tienen $n$ espacios de estacionamiento ubicados en l\'{\i}nea, nume\-rados en orden de $1$ a $n$, donde el espacio n\'umero $1$ es el m\'as cercano a la puerta de entrada del parqueadero, y el n\'umero $n$ se encuentra llegando a la salida. Una fila de $n$ autos se dispone a entrar al parqueadero. Cada uno de los conductores tiene un estacionamiento preferido  
que desea utilizar. Es posible representar todas las elecciones de los autos mediante una lista de n\'umeros $(a_1,a_2,\ldots,a_n)$, donde el auto $i$ escoge el lugar $a_i$.
Es posible que varios autos tengan el mismo espacio preferido; es decir, que $a_i = a_j$ para $i \neq j$.

Una vez llegan los autos en orden, cada uno de ellos se dirige al sitio que escogi\'o. Si el espacio est\'a vac\'{\i}o, el auto se estaciona en ese lugar. En caso contrario, el auto sigue andando y se ubica en el primer lugar vac\'{\i}o que encuentre. Si ninguno de los lugares siguientes est\'a libre, el auto no podr\'a estacionarse. 

\begin{defi} Una \emph{funci\'on de parqueo} de longitud $n$ es una lista $(a_1,a_2,\ldots,a_n)$  de elecciones,  para la cual todos los autos se pueden estacionar. El conjunto de todas las funciones de parqueo de longitud $n$ se denotar\'a por $\Pk_n$.
\end{defi}

\begin{ej}
Para $n=4$, la lista $(2,1,4,1)$ es una funci\'on de parqueo. En tal caso, los primeros tres autos se estacionan en su lugar preferido, y el \'ultimo auto se estaciona en la tercera posici\'on. Por otro lado, la lista $(3,1,4,3)$ no es una funci\'on de parqueo, pues el \'ultimo auto no podr\'a parquear.
\end{ej}

Konheim y Weiss \cite{KW} demostraron los siguientes dos teoremas:

\begin{teor}
Una sucesi\'on $(a_1, \ldots, a_n)$ de enteros con $1 \leq a_i \leq n$ es una funci\'on de parqueo si y s\'olo si contiene 
al menos un 1, al menos dos n\'umeros menores o iguales a $2$, y en general, al menos $k$ n\'umeros menores o iguales a $k$, para cada $k$ entre 1 y $n$. 
\end{teor}

\begin{teor}\label{kw}
Existen en total $(n+1)^{n-1}$ funciones de parqueo de longitud $n$.
\end{teor}

Es posible demostrar el Teorema \ref{kw} por medio de los arreglos de hiperplanos, gracias a la cercana relaci\'on entre las funciones de parqueo y el arreglo de Shi. Teniendo en cuenta el Teorema \ref{rshi}, es suficiente dar una biyecci\'on entre las funciones de parqueo de longitud $n$ y las regiones del arreglo de Shi $\Shi_n$. A continuaci\'on vamos a describir una biyecci\'on $\lambda$. La Figura \ref{lambda} ilustra esta biyecci\'on para $n=3$.

\begin{figure}[htb]\centering
\includegraphics[width=8cm]{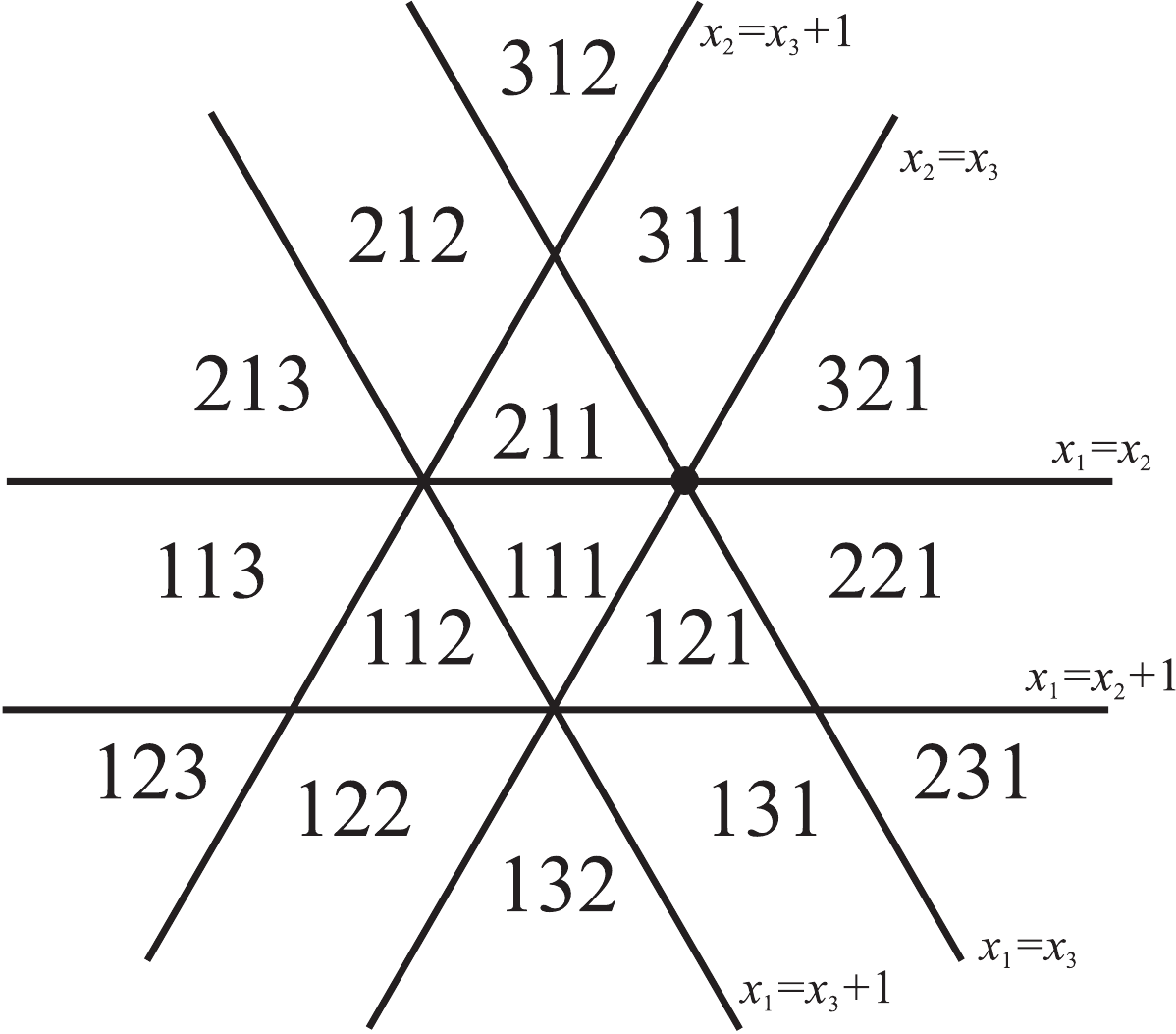}
\caption{El arreglo de Shi $\Shi_3$ y las etiquetas de sus regiones. El punto marcado es el origen.
}\label{lambda}
\end{figure}



Para definir la biyecci\'on $\lambda$, comenzaremos por llamar $R_0$ a la ``regi\'on base", en la cual $x_1<x_2<\cdots<x_n<x_1+1<x_2+1<\cdots<x_n+1$. Dadas dos regiones $R$ y $R'$, definimos la \emph{distancia} $d(R,R')$ como el n\'umero de hiperplanos $H$ del arreglo tales que $R$ y $R'$ se encuentran a lados opuestos de $H$; es decir, el n\'umero de hiperplanos que debemos cruzar para llegar de $R$ a $R'$. Construimos entonces la funci\'on $\lambda(R)$ recursivamente en funci\'on de $d(R_0,R)$ de la siguiente manera:

\begin{enumerate}
\item $\lambda(R_0)=(1,1,\ldots,1)\in \Pk_n$.
\item Si $R$ y $R'$ son regiones adyacentes, separadas \'unicamente por el hiperplano $x_i-x_j=0$, y $d(R_0,R') = d(R_0,R)+1$, definimos $\lambda(R')=\lambda(R)+e_i$, donde $e_i$ es el vector con s\'olo un 1 en la $i$-\'esima coordenada.
\item Si $R$ y $R'$ son regiones adyacentes, separadas \'unicamente por el hiperplano $x_i-x_j=1$, y $d(R_0,R') = d(R_0,R)+1$, definimos $\lambda(R')=\lambda(R)+e_j$. 
%
\end{enumerate}

Es f\'acil ver que la funci\'on $\lambda$ est\'a bien definida. Tambi\'en es cierto (pero no tan f\'acil) que es una biyecci\'on entre las regiones de $\Shi_n$ y las funciones de parqueo de longitud $n$.

La construcci\'on de esta biyecci\'on $\lambda$ es debida a Igor Pak, y puede ser extendida a las \emph{$k$-funciones de parqueo}, que est\'an relacionadas con los \emph{arreglos de $k$-Shi} $\Shi_n^k$. No entraremos en detalle al respecto, pero la demostraci\'on de la biyecci\'on de Pak y su generalizaci\'on a los arreglos de $k$-Shi puede encontrarse en \cite{S3}. 

El n\'umero $(n+1)^{n-1}$ juega un papel importante en varios contextos combinatorios. Tal vez el m\'as importante es que cuenta el n\'umero de \'arboles n\'umerados $0, \ldots, n$. Una \emph{inversi\'on} en un \'arbol es un par de v\'ertices $i,j$ con $1\leq i < j\leq n$ tal que $j$ se encuentra en el camino de $i$ hasta 0. Kreweras \cite{K} construy\'o una biyecci\'on entre las funciones de parqueo y los \'arboles numerados que demuestra el siguiente resultado:


\begin{teor}[Kreweras, \cite{K}]\label{kr}
El n\'umero de regiones del arreglo de Shi $\Shi_n$ que se encuentran a una distancia $k$ de la regi\'on base $R_0$ es igual al n\'umero de \'arboles con v\'ertices $0,\,1,\ldots,\,n$ que tienen ${n\choose 2}-k$ inversiones.\end{teor}

\subsection{\textsf{Arreglo de Linial}}

El siguiente arreglo es parecido a los anteriores, aunque el c\'alculo de  su polinomio caracter\'{\i}stico y su n\'umero de regiones, debidos a Alex Postnikov \cite{P}, es m\'as sutil. 

\begin{defi} 
El arreglo de Linial $\Lin_n$ en $\R^n$  est\'a formado por los hiperplanos
 $$ x_i-x_j= 1 \qquad \qquad (1\leq i< j\leq n).$$
\end{defi}

\begin{teor}\cite{P}
El polinomio caracter\'{\i}stico del arreglo de Linial $\Lin_n$ es 
$$\chi_{\Lin_n}(t)=\frac{t}{2^n}\sum_{k=0}^n  {n\choose k} (t-k)^{n-1}.$$
\end{teor}

Un  \emph{\'arbol alternante} es un \'arbol cuyos v\'ertices est\'an numerados $1, 2, \ldots, n$ de manera que todo v\'ertice es o bien mayor que todos sus vecinos o bien menor que todos ellos. 

\begin{teor}\cite{P}
El n\'umero de regiones del arreglo de Linial $\Lin_n$ es igual al n\'umero de \'arboles alternantes con $n+1$ v\'ertices, y est\'a dado por la f\'ormula $$\reg(\Lin_n)=\frac1{2^n}\sum_{k=0}^n  {n\choose k} (k+1)^{n-1}.$$
\end{teor}

Se conocen varias familias de objetos que est\'an en biyecci\'on con los \'arboles alternantes. Sin embargo, a\'un no se conoce una biyecci\'on natural entre las regiones de $\Lin_n$ y los \'arboles alternantes. Tampoco se conoce una interpretaci\'on de las regiones acotadas del arreglo de Linial,  en t\'erminos de \'arboles alternantes u otros objetos combinatorios. Tambi\'en ser\'{\i}a interesante encontrar una interpretaci\'on combinatoria para los coeficientes del polinomio $\chi_{\Lin_n}(t)$. Culminamos esta secci\'on  con un resultado muy sorprendente, la ``hip\'otesis de Riemann para el arreglo de Linial":

\begin{teor} \cite{P} Todas las ra\'{\i}ces del polinomio $\dfrac{\chi_{\Lin_n}(t)}t$ tienen  parte real igual a  $\dfrac n2$; es decir, son de la forma $\dfrac n2+bi$ con $b \in \R$.
\end{teor}


\section{\textsf{Arreglos complejos y el \'algebra de Orlik-Solomon}}

En esta secci\'on consideramos los arreglos de hiperplanos sobre el campo $\Cx$ de los n\'umeros complejos. Los complejos no forman un campo ordenado, por lo cual un hiperplano en $\Cx^d$ no tiene un lado `positivo' y  otro `negativo'. Por el contrario, el complemento de un hiperplano es un espacio conexo. Por lo tanto no es posible definir regiones como lo hicimos sobre los reales. Ahora la topolog\'{\i}a del complemento $\Cx^d - \A$ del arreglo es m\'as complicada, pero a\'un guarda una estrecha relaci\'on con el poset de intersecci\'on y el polinomio caracter\'{\i}stico, que se pueden definir igual que antes.

Sea $\A=\{H_1, \ldots, H_n\}$ un arreglo de hiperplanos en $\Cx^d$. Sea $E(\A)$ el \'algebra asociativa sobre $\Cx$ generada por los hiperplanos (considerados como s\'{\i}mbolos formales), sujeta \'unicamente a las relaciones $H_i^2=0$ para $1 \leq i \leq n$ y $H_iH_j=-H_jH_i$ para $1 \leq i < j \leq n$. Es claro que $\{\prod_{H \in \B} H \, : \, \B \subseteq \A\}$ es una base de $E(\A)$ como espacio vectorial, y por lo tanto $\dim(E(\A)) = 2^n$. (Algunos lectores reconocer\'an a $E(\A)$ como el \'algebra exterior de un espacio vectorial de dimensi\'on $n$.)

Definimos la funci\'on lineal $\partial:E(\A) \rightarrow E(\A)$ de la siguiente manera:
\[
\partial(H_{i_1}H_{i_2}\cdots H_{i_k}) = \sum_{j=1}^k (-1)^j H_{i_1} \cdots \widehat{H_{i_j}} \cdots H_{i_k}
\]
donde $\widehat{H_{i_j}}$ significa que omitimos el t\'ermino $H_{i_j}$. Decimos que $H_{i_1}, \ldots, H_{i_k}$ son \emph{dependientes} si sus vectores normales lo son; es decir, si $\dim(H_{i_1} \cap \cdots \cap H_{i_k}) > d-k$.

Sea $I_{\A}$ el ideal de $E(\A)$ generado por \\
$\bullet$ Los productos $H_{i_1}H_{i_2}\ldots H_{i_k}$ tales que $H_{i_1}\cap H_{i_2}\cap\ldots \cap H_{i_k}=\emptyset$, y\\
$\bullet$ Los elementos $\partial(H_{i_1}H_{i_2}\ldots H_{i_k})$ tales que $H_{i_1}, \ldots, H_{i_k}$ son dependientes.

\begin{defi}\cite{OS}\label{osalgebra}
El \emph{\'algebra de Orlik-Solomon} del arreglo $\A$ es $OS(\A)=E(\A)/I_{\A}$.
\end{defi}

Podemos considerar a $E(A) = E(\A)_0 \oplus E(\A)_1 \oplus \cdots \oplus E(\A)_n$  como un espacio vectorial graduado donde la componente $E(\A)_k$ de grado $k$ es generada por los productos $H_{i_1}\cdots H_{i_k}$ de grado $k$. Como $I_{\A}$ es un ideal homog\'eneo, el cociente $OS(\A) = OS(\A)_0 \oplus \cdots \oplus OS(\A)_n$ 
 tambi\'en es graduado. 

\begin{defi}
El \emph{polinomio de Hilbert} de $OS(A)$ se define como
$$\text{Hilb}_{OS(\A)}=\sum_{k=0}^{n}(dim(OS(\A)_k))x^k.$$
\end{defi}

\begin{figure}[htb]\centering
\includegraphics[width=4cm]{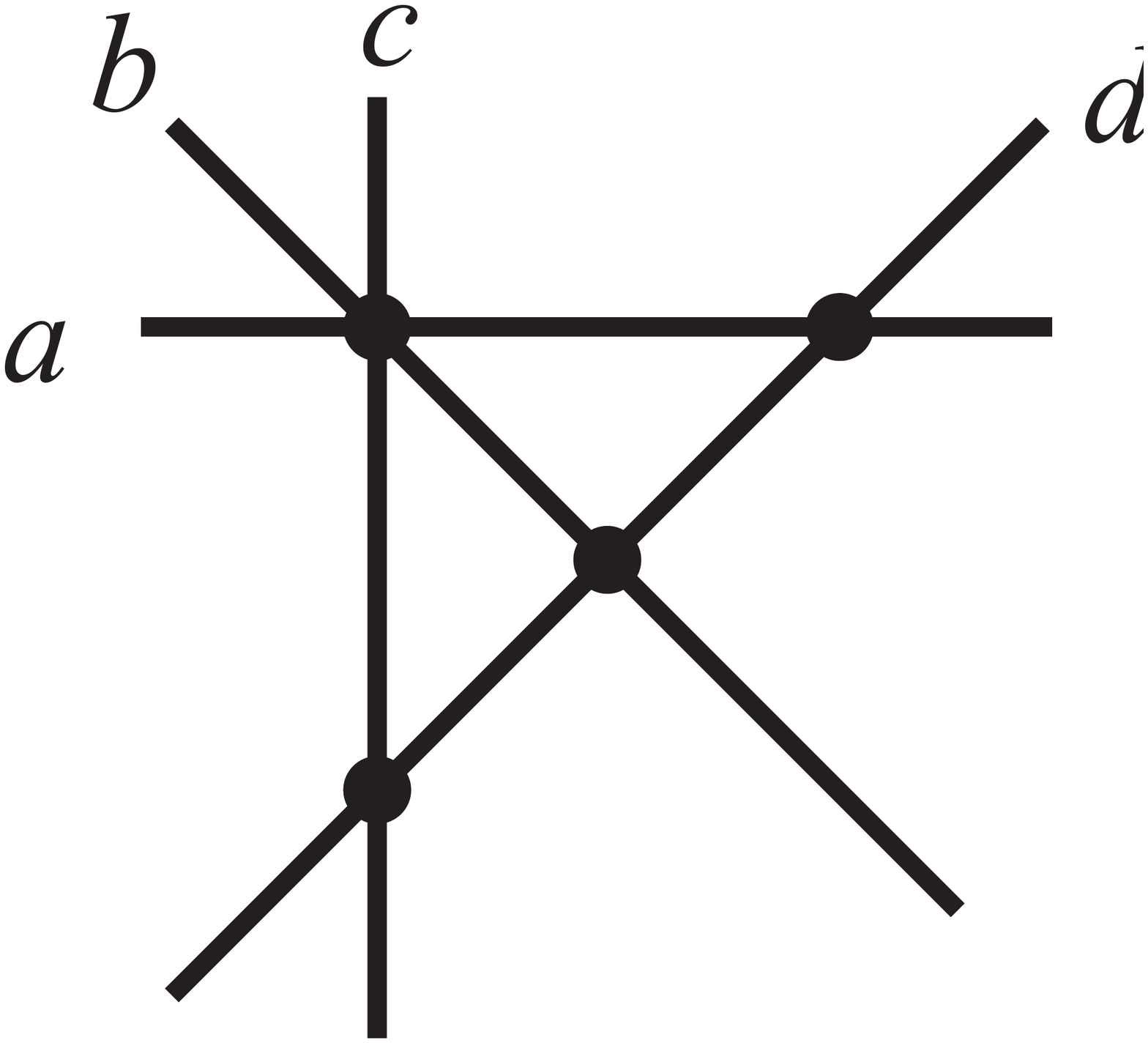}
\caption{Arreglo en $\R^2$}\label{os}
\end{figure}

\begin{ej}
Sea $\A$ el arreglo en $\Cx^2$ cuyos hiperplanos son $x=0, \, y=0,\, x+y=0,$ y $x-y=1$. En la Figura \ref{os} ilustramos la parte real de este arreglo. El lector puede calcular f\'acilmente el poset de intersecci\'on (que es igual si consideramos el arreglo en $\Cx^2$ o en $\R^2$), la funci\'on de M\"obius, y el polinomio caracter\'{\i}stico de $\A$, que es $\chi_\A(t) = t^2-4t+5$.

Ahora, $E(\A)$ es generado por variables no conmutativas $a,b,c,d$ sujetas a las relaciones
\[
a^2=b^2=c^2=d^2=0, \qquad ab=-ba, \, ac=-ca, \, \ldots,\,  cd=-dc.
\]
El \'algebra de Orlik-Solomon $OS(\A)$ se obtiene al introducir las relaciones adicionales
\[
abd=acd=bcd=0, \qquad -bc+ac-ab=0.
\]
De ah\'{\i} vemos que la siguiente es una base para $OS(\A)$ como espacio vectorial:
\[
\{1, \,\,\,\, a, b, c, d, \,\,\,\, ab, ac, ad, bd, cd\}.
\] 
No incluimos al monomio $bc$ en la base porque $bc=ac-ab$. Tampoco incluimos a $abc$ porque en $OS(\A)$ se tiene que $abc=a(ac-ab) = a^2(c-b)=0$.
Tenemos entonces que 
\[
\text{Hilb}_{OS(\A)}(x)=1+4x+5x^2.
\]
\end{ej}

Algunos lectores reconocer\'an alguna semejanza entre la Definici\'on \ref{osalgebra} y algunas construcciones en la topolog\'{\i}a algebraica. De hecho, la gran importancia del \'algebra de Orlik-Solomon se debe a los siguientes elegant\'{\i}simos resultados:

\begin{teor}\cite{OS}
La cohomolog\'{\i}a de de Rham de $\Cx^d-\A$ es isomorfa al \'algebra de Orlik-Solomon de $\A$:
\[
H^{\ast}_{DR}(\Cx^d - \A,\Cx)\cong OS(\A).
\]
\end{teor}

\begin{teor}\cite{OS}
El polinomio de Hilbert de $OS(\A)$ es
\[
\text{Hilb}_{OS(\A)}(x)= x^d \chi_{\A}(-1/x);
\]
es decir, el $i$-\'esimo n\'umero de Betti de $\Cx^d-\A$ es igual al valor absoluto del $(d-i)$-\'esimo coeficiente del polinomio caracter\'{\i}stico de $\A$.
\end{teor}

No daremos una definici\'on precisa de la cohomolog\'{\i}a de de Rham y los n\'umeros de Betti en estas notas, cuyo enfoque es combinatorio. Para ver una descripci\'on completa, referimos al lector a \cite{BT}, o a \cite{OT} para una presentaci\'on desde el punto de vista de los arreglos de hiperplanos. Nos limitaremos a mencionar que, como dijimos anteriormente, \textbf{en un espacio real $\R^d$}, el complemento de un arreglo $\A$ es una uni\'on disjunta de regiones contr\'actiles. Topol\'ogicamente es un espacio bastante trivial. Su \'unica caracter\'{\i}stica interesante es el n\'umero de regiones $\reg(\A)$, que ya sabemos calcular combinatoriamente. Por el contrario, \textbf{en un espacio complejo} $\Cx^d$, la topolog\'{\i}a del complemento $\Cx^d-\A$ es m\'as interesante. Por ejemplo, vemos que este es un espacio conexo, ya que podemos dar una vuelta alrededor de cada hiperplano. La cohomolog\'{\i}a del complemento es un anillo graduado que ``mide" de cierta manera la topolog\'{\i}a de este espacio. El $i$-\'esimo  n\'umero de Betti $\beta_i = \dim H^i_{DR}(\Cx^d - \A, \Cx)$ nos dice cu\'antos huecos $i$-dimensionales independientes tiene este espacio. Es de gran inter\'es que Orlik y Solomon hayan logrado dar una presentaci\'on puramente combinatoria de la cohomolog\'{\i}a de este espacio y su polinomio de Hilbert.

%

En resumen, el polinomio caracter\'{\i}stico $\chi_\A(x)$ de un arreglo de hiperplanos $\A$ es un poderoso invariante combinatorio, que contiene una gran cantidad de informaci\'on sobre el arreglo:
\begin{itemize}
\item Si el arreglo est\'a en $\R^n$, el complemento $\R^n-\A$ es una uni\'on de $|\chi_{\A}(-1)|$  componentes conexas, y $|\chi_{\A}(1)|$ de estas regiones son acotadas.
\item Si el arreglo est\'a en $\F_q^n$, el complemento $\F_q^n-\A$ tiene exactamente $\chi_{\A}(q)$ puntos.
\item Si el arreglo est\'a en $\Cx^n$, los n\'umeros de Betti del complemento $\Cx^n-\A$ son precisamente los coeficientes de $\chi_{\A}(q)$.
\end{itemize}

\'Este es tan s\'olo el comienzo de la teor\'{\i}a de arreglos de hiperplanos. Invitamos al lector a seguir profundizando en esta fascinante \'area.

\section*{\textsf{Agradecimientos.}} Agradecemos muy especialmente a Richard Stanley, de quien aprendimos la teor\'{\i}a combinatoria de arreglos de hiperplanos. A esto se debe que esta lecci\'on, escrita originalmente en 2003, tiene bastante en com\'un con sus notas \cite{S1}  de 2005.


\begin{thebibliography}{nn}

\bibitem{ALRS1}
F. Ardila, E. Le\'on, M. Rosas, y M. Skandera. Tres lecciones en combinatoria algebraica. I. Matrices totalmente no negativas y funciones sim\'etricas.

\bibitem{ALRS2}
F. Ardila, E. Le\'on, M. Rosas, y M. Skandera. Tres lecciones en combinatoria algebraica. II. Las funciones sim\'etricas y la teor\'{\i}a de las representaciones.

\bibitem{ALRS3}
F. Ardila, E. Le\'on, M. Rosas, y M. Skandera. Tres lecciones en combinatoria algebraica. III. Arreglos de hiperplanos. 2013

\bibitem{A} 
C. A. Athanasiadis. Characteristic polynomials of subspace arrangements and finite fields, \emph{Advances in Math.} {\bf 122} (1996), 193-233.

\bibitem{BT}
R. Bott y L. Tu. Differential Forms in Algebraic Topology, Berlin, New York: Springer-Verlag. 1982.

\bibitem{CR} 
H. Crapo y G.-C. Rota. On the Foundations of Combinatorial Theory: Combinatorial Geometries. MIT Press, Cambridge, MA, 1970.

\bibitem{KW}
A. G. Konheim y B. Weiss, An occupancy discipline and applications. \emph{SIAM J. Applied Math.} {\bf 14} (1966) 1266-1274.

\bibitem{K} G. Kreweras, \textit{Une famille de polyn\^omes ayant plusieurs propri\'et\'es \'enumeratives},\\ Periodica Mathematica Hungarica Vol 11(4), 1980.

\bibitem{tg} E. Le\'on, \textit{Conteos en Arreglos de Hiperplanos, N\'umeros de Catalan y Funciones de parqueo.} Universidad Nacional de Colombia, Bogot\'a, trabajo de grado, 2006.

\bibitem{OS} P. Orlik y L. Solomon, Combinatorics and topology of complements of hyperplanes. \emph{Invent. Math.} {\bf 56} (1980) 167-189.

\bibitem{OT} P. Orlik y H. Terao, Arrangements of Hyperplanes, Springer-Verlag, Berlin, 1992.

\bibitem{P} A. Postnikov y R. Stanley, \textit{Deformation of Coxeter Hyperplane Arrangements}\\ Massachusetts Institute of technology, Cambridge, MA 02139, 1997.

\bibitem{R} G. C. Rota. \textit{On the Foundations of Combinatorial Theory I: Theory of M\"obius Functions.} Zeitschrift f\"ur Wahrscheinlichkeitstheorie und Verwandte Gebiete, {\bf 2} (1964) 340-368.

\bibitem{So} A. Sokal. Bounds on the Complex Zeros of (Di)Chromatic Polynomials and Potts-Model Partition Functions.  \emph{Combin. Probab. Comput.} {\bf 10} (2001) 41-77.

\bibitem{So2} 
A. D. Sokal, Chromatic roots are dense in the whole complex plane. \emph{Combin. Probab. Comput.} {\bf 13} (2004), 221-261.

\bibitem{S} R. Stanley, \textit{Enumerative Combinatorics, vol 1.} Wadsworth \& Brooks Cole, Belmont, CA, 1986.

\bibitem{Sv2} R. Stanley, \textit{Enumerative Combinatorics, vol 2.} Cambridge University Press, Cambridge, 1999.

\bibitem{S1} R. Stanley, \textit{An Introduction to Hyperplane Arrangements.} IAS/ Park City Mathematics Series, 2005.

\bibitem{S3} R. Stanley. \textit{Hyperplane arrangements, parking functions, and tree inversions}, in Mathematical Essays in Honor of Gian-Carlo Rota (B. Sagan and R. Stanley, eds.) Birkh\"auser, Boston/Basel/Berlin, 1998, pp. 259-375.

\bibitem{S4} R. Stanley, \textit{Hyperplane arrangements, interval orders, and trees.} \emph{Proc. Nat.
Acad. Sci.} {\bf 93} (1996), 2620-2625.


\bibitem{zas} T. Zaslavsky, \textit{Facing up to Arrangements: Face-Count Formulas for Partitions of Space by Hyperplanes}, Thesis (MIT, 1974) and \emph{Mem. Amer. Math. Soc.}, No. 154, Amer. Math. Soc., Providence, R.I., 1975. 
\end{thebibliography}
\end{document}